\documentclass[11pt,reqno]{amsart}
\usepackage{amsmath,amsthm,amscd,amsfonts,amssymb,color}
\usepackage{cite}
\usepackage[mathscr]{eucal}
\usepackage[bookmarksnumbered,colorlinks,plainpages]{hyperref}
\newtheorem{theorem}{Theorem}[section]
\newtheorem{lemma}[theorem]{Lemma}

\theoremstyle{definition}
\newtheorem{definition}[theorem]{Definition}
\newtheorem{example}[theorem]{Example}
\theoremstyle{remark}

\numberwithin{equation}{section}
\def\DJ{\leavevmode\setbox0=\hbox{D}\kern0pt\rlap
 {\kern.04em\raise.188\ht0\hbox{-}}D}

\input{tcilatex}

\begin{document}

\title[Positive solution via fixed point results]{Positive solution to fractional thermostat model in Banach spaces via fixed point results}
\author[H.\ Garai, L.K. \ Dey, A. \ Chanda]
{Hiranmoy Garai$^{1}$,  Lakshmi Kanta Dey$^{2}$, Ankush Chanda$^{3}$}

\address{{$^{1}$\,} Hiranmoy Garai,
                    Department of Mathematics,
                    National Institute of Technology
                    Durgapur, India.}
                    \email{hiran.garai24@gmail.com}
\address{{$^{2}$\,} Lakshmi Kanta Dey,
                    Department of Mathematics,
                    National Institute of Technology
                    Durgapur, India.}
                    \email{lakshmikdey@yahoo.co.in}
\address{{$^{3}$\,} Ankush Chanda,
                    Department of Mathematics,
                    National Institute of Technology
                    Durgapur, India.}
                    \email{ankushchanda8@gmail.com}


\begin{abstract}
The motive behind this manuscript is to set up the existence and uniqueness of a positive solution for a fractional thermostat model for certain values of the parameter $\lambda>0$. We accomplish sufficient
conditions for the existence of a positive solution to the model, and afterwards formulate a non-trivial example to authenticate the grounds of our obtained results. Our findings are based on certain fixed point results of contractions depending on couple of altering distance functions $\phi$ and $\psi$ in the setting of Banach spaces that are discussed in this sequel.
\end{abstract}
 
\maketitle

\setcounter{page}{1}

\centerline{}

\centerline{}

\section{\bf Introduction and Preliminaries}
Metric fixed point theory is extensively employed in different mathematical branches as well as in real world problems originating in applied sciences. The results on fixed points of contractive maps considered on different underlying spaces are mostly applied on the validation of the existence and uniqueness of solutions of functional, differential or integral equations. The plurality of these types of problems elicits the probe for more and better techniques, which is a salient feature of the recent research works in this literature.

The dawning of fixed point theory on a complete metric space is integrated with the Banach contraction principle due to S. Banach \cite{B1}.
\begin{theorem} \label{thm1}
Let $(X,d)$ be a complete metric space and $T$ be a self-mapping on $X$ satisfying  $$d(Tx,Ty)\leq kd(x,y)$$ for all $x,y \in X$ and $k\in [0,1)$.
Then $T$ has a unique fixed point $z\in X$, and for any $x\in X$ the sequence of iterates $(T^nx)$ converges to $z$.
\end{theorem}
Because of its inferences and huge usability in mathematical theory, Banach contraction principle has been improved and generalized in metric spaces, partially ordered metric spaces, Banach spaces and many other spaces, see\cite{AEO,AE,AS,BA,C,C1,CSV,C2,E,MK}.

In $1962$, E. Rakotch \cite{R2} proved that the Theorem \ref{thm1} still holds if the constant $k$ is replaced by a contraction monotone decreasing function. He proved the following theorem as a corollary.
\begin{theorem}
Let $(X,d)$ be a complete metric space and $T:X\to X$ be a mapping such that $$d(Tx,Ty) \leq \alpha(x,y)d(x,y)$$ for all $x,y \in X$, where $\alpha$ is a function defined on  $[0,\infty)$ satisfying the following conditions:
\begin{enumerate}
\item[i)] 
$\alpha(x,y)=\alpha(d(x,y)),~~ i.e.,~~ \alpha$ is dependent on the distance of $x$ and $y$ only;
\item[ii)] 
$0\leq \alpha(\tau) <1$ for all $\tau >0$;
\item[iii)] 
$\alpha(\tau)$ is monotonically decreasing function of $\tau.$
\end{enumerate} 
Then $T$ has a unique fixed point.
\end{theorem}

In his research article, D.S. Jaggi \cite{J} used the continuity and some different contractive conditions on the mapping to attain the succeeding result.
\begin{theorem}
Let $f$ be a continuous self-map defined on a complete metric space $(X,d)$. Further let, $f$ satisfies the following condition:
$$d(f(x),f(y)) \leq \frac{\alpha d(x,f(x))d(y,f(y))}{d(x,y)} + \beta d(x,y)$$
for all $x,y \in X,$ with $x\neq y$ and for some $\alpha, \beta \in [0,1) $ with $\alpha + \beta <1$. Then $f$ has a unique fixed point in $X$.
\end{theorem}

Thereafter, Khan et al. \cite{KSS} extended and generalized the Banach principle using a control function, known as altering distance function. 
\begin{definition} \label{adf}
A function $\varphi:[0,\infty)\to [0,\infty)$ is called an altering distance function if it satisfies the following conditions:
\begin{enumerate}
\item[i)] 
$\varphi$ is monotone increasing and continuous;
\item[ii)] 
$\varphi(t)=0$ if and only if $t=0$.
\end{enumerate}
\end{definition}
 In \cite{KSS}, the authors also proved the following fixed point theorem by means of the newly originated concept of control functions.
\begin{theorem} \label{thm1.1}
Let $(X,d)$ be a complete metric space and $\psi:[0,\infty)\to [0,\infty)$. Also suppose that $f:X\to X$ is a mapping satisfying $$\psi(d(fx,fy)) \leq a\psi(d(x,y))$$ for all $x,y \in X$ and for some $0\leq a<1$. Then $f$ has a unique fixed point.
\end{theorem}

Alber and Guerre-Delabriere \cite{AG} introduced the notion of weak contractions in a Hilbert space. 
\begin{definition} \cite{AG} 
Let $(X,d)$ be a metric space. A mapping $T:X \to X$ is called weakly contractive if and only if
$$d(Tx,Ty) \leq d(x,y) - \phi(d(x,y))$$ for all $x,y \in X$, where $\phi$ is an altering distance function.
\end{definition} 
Afterwards, Rhoades \cite{R1} generalized  the weak contraction condition in metric spaces and proved the following  fixed point result in complete metric spaces.
\begin{theorem} \label{thm1.2}
Let $(X,d)$ be a complete metric space. If $T:X \to X$ is a weakly contractive map, then $T$ has a unique fixed point.
\end{theorem}

In their research paper, Dutta and Choudhury \cite{DC} generalized Theorem \ref{thm1.1} and \ref{thm1.2} to obtain the following theorem.
\begin{theorem}
Let $(X,d)$ be a complete metric space and $T:X\to X$ be a mapping satisfying $$\psi(d(Tx,Ty)) \leq \psi(d(x,y)) - \phi(d(x,y))$$ for all $x,y \in X$, where $\psi ~~ and ~~ \phi$ are two altering distance functions. Then $T$ has a unique fixed point.
\end{theorem}
Fractional calculus has been explored for many decades mostly as a pure analytic mathematical
branch. Though in recent times, many authors are showing  a lot of interest in its applications for solving ordinary differential equations. Fractional differential equations appear in different engineering and scientific branches as the mathematical modelling of systems and techniques in the domains of physics, chemistry, aerodynamics, robotics and many more. For a few recent articles in this direction, see \cite{CHS,GAPV,NP,BL1,CW,DG,SD} and the references in that respect.

Considering exclusively positive solutions are effective for several applications, inspired by the aforementioned works, in our draft, we set up an existence and uniqueness theorem to find a positive solution for a fractional thermostat model with a positive parameter. With a view to inspect the solution, we enquire into some new fixed point results in a Banach space by considering a pair of altering distance functions in a more adequate appearance. We also extend our results in a Banach space which is equipped with an arbitrary binary relation and keeps the order preserving property of the mappings. Finally, a suitable non-trivial example is furnished to substantiate the effectiveness of our results.

\section{\bf Fixed Point Results}
This section deals with the results on the existence and uniqueness of fixed points of maps satisfying a  contractive condition with a pair of control functions in a Banach space and also their proofs. Moreover, we formulate an example to elucidate our attained results.

\begin{theorem} \label{thm2}
Let $(X,\|.\|)$ be a Banach space and $C$ be a closed subset of $X$. Let $ T : C \to C$ be a mapping. Assume that, there exist two altering distance functions $\phi, \psi:[0,\infty)\to[0,\infty)$ such that 
\begin{eqnarray}\label{eq1}
\phi(\|Tx-Ty\|)&\leq & \phi(\|Tx-y\|)-\psi(\|x-y\|)
\end{eqnarray}
for all $x,y \in C$. Then $T$ has a unique fixed point in $C$.
\end{theorem}
\begin{proof}
Let $x_0\in C$ be arbitrary but fixed. Consider, the iterated sequence $\{x_n\}$ where $x_n=T^nx_0$ for each natural number $n$.

Therefore, by given condition we have,
$$\phi(\|Tx_{n-1}-Tx_{m-1}\|)\leq \phi(\|Tx_{n-1}-x_{m-1}\|)-\psi(\|x_{n-1}-x_{m-1}\|)$$
\begin{equation} \label{eqn1} 
\Rightarrow \phi(\|x_n-x_m\|)\leq \phi(\|x_n-x_{m-1}\|)-\psi(\|x_{n-1}-x_{m-1}\|)
\end{equation}
which implies that,
$$\phi(\|x_n-x_m\|)\leq \phi(\|x_n-x_{m-1}\|)$$
for all $ n,m\in \mathbb{N}$.
Since $\phi$ is monotone increasing, we have
$$\|x_n-x_m\|\leq \|x_n-x_{m-1}\|$$
for all $n,m\in \mathbb{N}$.

Interchanging the role of $x_n$ and $x_m$ in above equation  we get, 
\begin{eqnarray} \label{eq2} 
\|x_n-x_m\| & \leq & \|x_{n-1}-x_m\|
\end{eqnarray}
for all $n,m\in \mathbb{N}$.

Thus for each fixed $n\in \mathbb{N}$, we can conclude that the sequence $\{s^{(n)}_m\}_{m\in \mathbb{N}}$ of non-negative real numbers is monotone decreasing, where, $s^{(n)}_m=\|x_n-x_m\|$ for each $m\in \mathbb{N}$. So, $\{s^{(n)}_m\}_{m\in \mathbb{N}}$ is convergent for each $n\in \mathbb{N}$.

Let, $$ \lim_{m\to \infty}s^{(n)}_m=a^{(n)}$$ for each $n\in \mathbb{N}.$

Now from equation \ref{eqn1}, we have,
$$ \phi(\|x_n-x_m\|)+\psi(\|x_{n-1}-x_{m-1}\|)\leq \phi(\|x_n-x_{m-1}\|).$$ 
Keeping $n$ fixed, taking limit as $m\to \infty$ in both sides of above in-equation and using the continuity of $\phi, \psi$ on $[0,\infty)$, we get
$$\displaystyle \lim_{m\to \infty}\phi(\|x_n-x_m\|)+\displaystyle \lim_{m\to \infty}\psi(\|x_{n-1}-x_{m-1}\|) \leq \displaystyle \lim_{m\to \infty}\phi(\|x_n-x_{m-1}\|) $$
$$\Rightarrow\phi(\displaystyle \lim_{m\to \infty}\|x_n-x_m\|)+\psi(\displaystyle \lim_{m\to \infty}\|x_{n-1}-x_{m-1}\|) \leq \phi(\displaystyle \lim_{m\to \infty}\|x_n-x_{m-1}\|) $$
\begin{eqnarray*} 
& \Rightarrow & \phi(a^{(n)})+\psi(a^{(n-1)}) \leq \phi(a^{(n)})\\
& \Rightarrow & \psi(a^{(n-1)}) \leq 0\\
& \Rightarrow & \psi(a^{(n-1)}) = 0\\
& \Rightarrow & a^{(n-1)}=0 \mbox{ [since, $\psi(t)=0$ if and only if $t=0$]}.
\end{eqnarray*}
Therefore $a^{(n)}=0$ for all $n\in \mathbb{N}$, i.e., $\displaystyle \lim_{m\to \infty}\|x_n-x_m\|=0$ for all $n\in \mathbb{N}$.

Now, we consider the sequence of functions $\{f_m\}$ defined on $C$ by,
\[f_m(x)= \left\{\begin{array}{lr}
\|x_n-x_m\|, & \mbox {if $x =x_n$ for some $n\in \mathbb{N}$};\\
0, & \mbox{otherwise}.
\end{array}
\right.\]

Therefore, $\displaystyle \lim_{m\to \infty}f_m(x)=0$ for all $x\in C.$ Thus the limit function $f$ of the sequence of functions $\{f_m\}$ is given by $$f(x)=0 \mbox{ for all }x\in C.$$

Now, let $$M_m=\displaystyle \sup_{x\in C}|f_m(x)-f(x)|.$$

Therefore,
\begin{eqnarray*}
M_m &=& \displaystyle \sup_{x\in C}|f_m(x)| \mbox{  [since, $f(x)=0$ for all $x\in C$]}\\
&=& \displaystyle \sup_{n}|f_m(x_n)|\\
&=& \displaystyle \sup_{n}\|x_n-x_m\|.
\end{eqnarray*}

But, we know from \ref{eq2} that,
$$\|x_n-x_m\| \leq \|x_{n-1}-x_m\| \leq \|x_{n-2}-x_m\| \leq ... \leq \|x_1-x_m\|, $$ 
which implies that
\begin{eqnarray*}
& &\displaystyle \sup_{n}\|x_n-x_m\| \leq  \|x_1-x_m\|\\
& \Rightarrow & M_m \leq  \|x_1-x_m\|\\
& \Rightarrow & \displaystyle \lim_{m\to \infty} M_m \leq \displaystyle \lim_{m\to \infty}\|x_1-x_m\|=0\\
& \Rightarrow & \displaystyle \lim_{m\to \infty} M_m=0.
\end{eqnarray*}
Let $\epsilon>0$ be arbitrary. Since, $ \displaystyle \lim_{m\to \infty} M_m=0$, so there exists a natural number $N$ such that  
\begin{eqnarray*}
& & |M_m|< \epsilon \mbox{ for all } m\geq N\\
& \Rightarrow & \displaystyle \sup_{x\in C} |f_m(x)-f(x)|<\epsilon \mbox{ for all } m\geq N\\ 
& \Rightarrow & |f_m(x)-f(x)|< \epsilon \mbox{ for all } m\geq N \mbox{ and for all } x\in C\\
& \Rightarrow & |f_m(x)|< \epsilon \mbox{ for all } m\geq N \mbox{ and for all } x\in C.
\end{eqnarray*}
In particular, we have $$|f_m(x_n)|< \epsilon \mbox{ for all } m\geq N \mbox{ and for all } n\in \mathbb{N}.$$
Therefore, we can write, 
\begin{eqnarray}\label{eq3}
\big{|}\|x_n-x_m\|-0\big{|} &<&\epsilon
\end{eqnarray}
for all $n,m \geq N$.

Next, we consider the double sequence $\{s_{nm}\}_{n,m \in \mathbb{N}}$ of real numbers, where $$s_{nm}=\|x_n-x_m\|$$ for all $n,m \in \mathbb{N}.$ Here using \ref{eq3}, we have 
$$|s_{nm}-0|< \epsilon \mbox{ for all } n,m \geq N.$$

This implies the double sequence $\{s_{nm}\}_{n,m \in \mathbb{N}}$ converges to $0$, i.e, $$\displaystyle \lim_{n,m\to \infty}\|x_n-x_m\|=0.$$

Thus, $\{x_n\}$ is a Cauchy sequence in $C$. $C$ being complete, $\{x_n\}$ must converge to some $z\in C.$

Now from \ref{eq1}, we have,
\begin{eqnarray*}
& & \phi(\|x_{n+1}-Tz\|)  \leq  \phi(\|x_{n+1}-z\|)-\psi(\|x_n-z\|)\\
& \Rightarrow & \phi(\|x_{n+1}-Tz\|)  \leq  \phi(\|x_{n+1}-z\|)\\
& \Rightarrow & \displaystyle \lim_{n\to \infty}\phi(\|x_{n+1}-Tz\|) \leq  \displaystyle \lim_{n\to \infty}\phi(\|x_{n+1}-z\|)\\
& \Rightarrow & \phi(\displaystyle \lim_{n\to \infty}\|x_{n+1}-Tz\|) \leq \phi(\displaystyle \lim_{n\to \infty}\|x_{n+1}-z\|) \\
& \Rightarrow & \phi(\displaystyle \lim_{n\to \infty}\|x_{n+1}-Tz\|) \leq \phi(0)=0\\
& \Rightarrow & \phi(\displaystyle \lim_{n\to \infty}\|x_{n+1}-Tz\|) =0\\
& \Rightarrow &\displaystyle \lim_{n\to \infty}\|x_{n+1}-Tz\|=0.
\end{eqnarray*}
The above equation shows that the sequence $\{x_n\}$ converges to $Tz$. Thus, $Tz=z$ and $z$ is a fixed point of $T$.

Finally, we check the uniqueness of the fixed point $z$. To check this, let $z_1$ be another fixed point of $T$, i.e., $Tz_1=z_1$.

From \ref{eq1} and using Definition \ref{adf}, we have
\begin{eqnarray*}
& & \phi(\|Tz-Tz_1\|)\leq \phi(\|Tz-z_1\|)-\psi(\|z-z_1\|)\\
& \Rightarrow & \phi(\|z-z_1\|)+\psi(\|z-z_1\|)\leq \phi(\|z-z_1\|)\\
& \Rightarrow & \psi(\|z-z_1\|)\leq 0\\
& \Rightarrow & \psi(\|z-z_1\|)=0\\
& \Rightarrow & z=z_1.
\end{eqnarray*}
Therefore, $z$ is the only fixed point of $T$.
\end{proof}

Now, we generalize Theorem \ref{thm2} in a Banach space which is equipped with an arbitrary binary relation and state the subsequent theorem.
\begin{theorem}
Let $(X,\|.\|)$ be a Banach space and $\mathscr{R}$ be an equivalence relation on $X$. Assume that $X$ has the property that if $\{x_n\}$ be any sequence in $X$ converging to $z\in X$, then $x_n \mathscr{R} z$ for each natural number $n$. Let $C$ be a closed subset of $X$  and $ T : C \to C$ be a mapping such that $T$ satisfies the following conditions:
\begin{enumerate}
\item[i)] 
$T$ is order-preserving;
\item[ii)] 
$\phi(\|Tx-Ty\|)\leq \phi(\|Tx-y\|)-\psi(\|x-y\|)$  for all $x,y\in C$ such that $x\mathscr{R}y$
\end{enumerate}
where $\phi, \psi :[0,\infty)\to [0,\infty)$ are two altering distance functions. 
 Then $T$ has a unique fixed point in $C$ if  there exists $x_0 \in X$ such that $x_0 \mathscr{R} Tx_0$.
\end{theorem}
\begin{proof}
The proof of this theorem is analogous to the previous one and so omitted.
\end{proof}
In next portion of this section, we present a result which not only gives the guarantee of existence of fixed point but also properly point out the fixed point.
\begin{theorem}
Let $(X,\|.\|)$ be a Banach space and $C$ be a closed subspace of $X$. Let $ T : C \to C$ be a mapping. Also assume that, there exist two altering distance functions $\phi, \psi :[0,\infty)\to [0,\infty)$ such that $T$ satisfies the following conditions:
\begin{enumerate}
\item[$(i)$] $\phi(\|Tx-Ty\|) \leq \phi(\|x-y\|)-\psi(\|x-y\|)$;
\item[$(ii)$] $\phi(\|Tx-y\|) \leq \phi(\|x-y\|)-\psi(\|x-y\|)$
\end{enumerate}
for all $x,y \in C$.
Then the null vector of $X$ is the only fixed point of $T$. 
\end{theorem}
\begin{proof}
Let $x_0 \in C$ be arbitrary but fixed and consider the iterated sequence $\{x_n\}$ where $x_n=T^nx_0$ for all $n \in \mathbb{N}$.

Let $s_n=\|x_n-x_{n+1}\|$ for all $n \in \mathbb{N}$.

Now, by condition $(i)$ we get
\begin{eqnarray*}
& &\phi(\|Tx_n-Tx_{n+1}\|) \leq \phi(\|x_n-x_{n+1}\|)-\psi (\|x_n-x_{n+1}\|)\\
& \Rightarrow & \phi(\|x_{n+1}-x_{n+2}\|) \leq \phi(\|x_n-x_{n+1}\|)\\
& \Rightarrow & \phi(s_{n+1}) \leq \phi(s_n)\\
& \Rightarrow &  s_{n+1} \leq s_n.
\end{eqnarray*}
This is true for all natural number $n$, which implies that $\{s_n\}$ is a decreasing sequence of non-negative reals and hence this sequence must converge. Let, $$\displaystyle \lim_{n \to \infty}s_n=a.$$
Again, from $(i)$ we have,
\begin{eqnarray*}
& & \phi(s_{n+1}) \leq \phi(s_n)-\psi(s_n)\\
& \Rightarrow & \displaystyle \lim_{n \to \infty}\phi(s_{n+1}) \leq \displaystyle \lim_{n \to \infty}\phi(s_n)-\displaystyle \lim_{n \to \infty}\psi(s_n)\\
& \Rightarrow & \phi(a) \leq \phi(a)-\psi(a)\\
& \Rightarrow & \psi(a) \leq 0\\
& \Rightarrow & \psi(a)=0\\
& \Rightarrow & a=0\\
& \Rightarrow & \displaystyle \lim_{n \to \infty}s_n=0.
\end{eqnarray*}
Therefore,
$$\displaystyle \lim_{n \to \infty}{\|x_n-x_{n+1}-\theta\|}=0.$$
This shows that the sequence $\{u_n\}$ in $C$ converges strongly to $\theta$, where $\theta$ is the null vector in $X$ and $u_n=x_n-x_{n+1}$ for all natural numbers $n$.
Now,
\begin{eqnarray*}
& & \phi(\|Tu_n-T\theta\|) \leq \phi(\|u_n-\theta\|)-\psi(\|u_n-\theta\|)\\
& \Rightarrow & \displaystyle \lim_{n \to \infty}\phi(\|Tu_n-T\theta\|) \leq \displaystyle \lim_{n \to \infty}\phi(\|u_n-\theta\|)-\displaystyle \lim_{n \to \infty}\psi(\|u_n-\theta\|)\\
& \Rightarrow & \displaystyle \lim_{n \to \infty}\phi(\|Tu_n-T\theta\|) \leq 0\\
& \Rightarrow & \displaystyle \lim_{n \to \infty}\phi(\|Tu_n-T\theta\|)=0\\
& \Rightarrow & \phi(\displaystyle \lim_{n \to \infty}\|Tu_n-T\theta\|)=0\\
& \Rightarrow & \displaystyle \lim_{n \to \infty}\|Tu_n-T\theta\|=0.\\
\end{eqnarray*}
Again, by condition $(ii)$ we get,
\begin{eqnarray*}
& & \phi(\|Tu_n-\theta\|) \leq \phi(\|u_n-\theta\|)-\psi(\|u_n-\theta\|)\\
& \Rightarrow & \displaystyle \lim_{n \to \infty}\phi(\|Tu_n-\theta\|) \leq \displaystyle \lim_{n \to \infty}\phi(\|u_n-\theta\|)-\displaystyle \lim_{n \to \infty}\psi(\|u_n-\theta\|)\\
& \Rightarrow & \displaystyle \lim_{n \to \infty}\phi(\|Tu_n-\theta\|) \leq 0\\
& \Rightarrow & \displaystyle \lim_{n \to \infty}\phi(\|Tu_n-\theta\|)=0\\
& \Rightarrow & \phi(\displaystyle \lim_{n \to \infty}\|Tu_n-\theta\|)=0\\
& \Rightarrow & \displaystyle \lim_{n \to \infty}\|Tu_n-\theta\|=0.\\
\end{eqnarray*}
Therefore by the uniqueness of limit, we obtain, $$T\theta=\theta,$$ i.e., $\theta$ is a fixed point of $T$.

Finally, suppose $z$ be another fixed point of $T$. Therefore,
\begin{eqnarray*}
& & \phi(\|Tz-T\theta\|) \leq \phi(\|z-\theta\|)-\psi(\|z-\theta\|)\\
& \Rightarrow & \phi(\|z-\theta\|) \leq \phi(\|z-\theta\|)-\psi(\|z-\theta\|)\\
& \Rightarrow & \psi(\|z-\theta\|)\leq 0\\
& \Rightarrow & \psi(\|z-\theta\|)=0\\
& \Rightarrow & z=\theta.
\end{eqnarray*}
Therefore $\theta$ is the only fixed point of $T$ in $C$.
\end{proof}
\begin{example}
Consider the Banach space $\mathbb{R}$ endowed with the usual norm and define a relation $\mathscr{R}$ on $\mathbb{R}$ by:
for $x,y \in \mathbb{R}$  $x \mathscr{R} y$ if and only if either $x,y \in[-(n+1),-n]$ or $x,y \in [n,n+1]$ for some $n\in \mathbb{N}$ or $x=y=0$. Then clearly $\mathscr{R}$ is an equivalence relation on $\mathbb{R}$.

Now, let  $C=C_1 \cup C_2 \cup C_3$, where $C_1=[-2,-1],~~ C_2=[1,2], ~~ C_3=\{0\}$. Then $C$ is a closed subset of $\mathbb{R}$.

Define, a mapping $T:C \to C$ by 
\[   Tx = \left\{\begin{array}{ll}
        -x, & \text{if } x\in C_1;\\
        0, & \text{if } x\in C_2 \cup C_3.
        \end{array}\right.  \]
Therefore, 
\[   \|Tx-Ty\| = \left\{\begin{array}{ll}
        |x-y|, & \text{if } x,y\in C_1;\\
        |x|, & \text{if } x\in C_1 \text{ and  } y\in C_2 \cup C_3;\\
        0, & \text{if }  x,y \in C_2 \cup C_3.
        \end{array}\right.  \]
\[   \|Tx-y\| = \left\{\begin{array}{ll}
        |x+y|, & \text{if } x\in C_1;\\
        |y|, & \text{if } x\notin C_1.
        \end{array}\right.  \]
Consider the functions $\phi, \psi:[0,\infty)\to [0,\infty)$ defined by
\begin{eqnarray*}
& & \phi(t)=t^2 \\
& & \psi(t)=\frac{t^2}{100000}
\end{eqnarray*}
for all $t\in [0,\infty)$.

Then, clearly $\phi, \psi$ are two altering distance functions. Let, $x,y \in C$ be arbitrary such that $x\mathscr{R} y$. Then the following cases arise.

\text{Case $I$:}  Let $x,y \in C_1$. Then,
\begin{eqnarray*}
& & \phi(\|Tx-Ty\|)+\psi(\|x-y\|)-\phi(\|Tx-y\|)\\
&=& |x-y|^2+\frac{|x-y|^2}{100000}-|x+y|^2\\
&=& -4xy + \frac{(x-y)^2}{100000}\\
&\leq & 0\\
\Rightarrow  \phi(\|Tx-Ty\|)& \leq & \phi(\|Tx-y\|)-\psi(\|x-y\|). 
\end{eqnarray*}
\text{Case $II$:}  Let $x,y \in C_2$. Then,
\begin{eqnarray*}
& & \phi(\|Tx-Ty\|)+\psi(\|x-y\|)-\phi(\|Tx-y\|)\\
&=& 0+ \frac{(x-y)^2}{100000} - y^2\\
&\leq & 0\\
\Rightarrow  \phi(\|Tx-Ty\|)& \leq & \phi(\|Tx-y\|)-\psi(\|x-y\|).
\end{eqnarray*}
\text{Case $III$:}  Let $x,y \in C_3$. Then clearly the equality holds.

Thus,
$$ \phi(\|Tx-Ty\|) \leq  \phi(\|Tx-y\|)-\psi(\|x-y\|) $$
for all $x,y \in C$ with $x \mathscr{R} y$.

Also it is easily seen that $T$ is order preserving and $0$ is the only fixed point of $T$.

\end{example}

\section{\bf Application to Fractional Thermostat Model}
The motivation of this section is to provide an application of the results discussed in this manuscript. For this purpose, we consider the following fractional thermostat model 
\begin{equation} \label{equ2}
^CD^{\alpha}u(t) + \lambda f(t,u(t))=0, t\in [0,1],
\end{equation}
subject to the boundary conditions:
\begin{equation} \label{equ3}
u'(0)=0, \beta ^CD^{\alpha - 1}u(1) + u(\eta)=0,
\end{equation}
where $^CD^{\alpha}$ stands for Caputo fractional derivative  of order $\alpha$, $\lambda$ is a positive constant and $1 <\alpha \leq 2,~~ 0 \leq \eta \leq 1,~~ \beta >0$ such that the following conditions hold:
\begin{enumerate}
\item $\beta \Gamma(\alpha)-(1-\eta)^{(\alpha -1)}>0$;
\item $f: [0,1] \times \mathbb{R} \to \mathbb{R}^+ $ is a continuous function;
\item $u:[0,1] \to \mathbb{R}$ is continuous.
\end{enumerate}
Our aim is to derive some sufficient conditions under which the problem \ref{equ2} with the boundary conditions \ref{equ3} possesses unique positive solution for certain values of the parameter $\lambda$. To proceed further, we first recall the following lemmas.
\begin{lemma} \cite{NP} \label{lem1}
Assume $f\in C[0,1]$. A function $u\in C[0,1]$ is a solution of the boundary value problem 
\begin{equation}
^CD^{\alpha}u(t) + \lambda f(t,u(t))=0, t\in [0,1],
\end{equation}
\begin{equation}
u'(0)=0, \beta ^CD^{\alpha - 1}u(1) + u(\eta)=0,
\end{equation}
if and only if it satisfies the integral equation
\begin{equation}
u(t)= \int_{0}^{1}G(t,s)f(s)ds
\end{equation}
where $G(t,s)$ is the Green's function (depending on $\alpha$) given by
$$G(t,s)= \beta + H_{\eta}(s) - H_t(s)$$
and for $r \in [0,1],$  $H_r(s):[0,1] \to \mathbb{R}$ is defined as $H_r(s)= \frac{(r-s)^{\alpha - 1}}{\Gamma(\alpha)} \mbox{  for } s \leq r$ and $H_r(s)=0 \mbox{  for  } s> r$, i.e., 
\[G(t,s) = \left\{ \begin{array}{lr}
\beta - \frac{(t-s)^{\alpha - 1}}{\Gamma(\alpha)} + \frac{(\eta-s)^{\alpha - 1}}{\Gamma(\alpha)}, & \mbox{if $0\leq s \leq \eta,~~ s\leq t$};\\
\beta + \frac{(\eta-s)^{\alpha - 1}}{\Gamma(\alpha)}, & \mbox{if $0\leq s \leq \eta,~~ s\geq t$};\\
\beta - \frac{(t-s)^{\alpha - 1}}{\Gamma(\alpha)}, & \mbox{if $\eta\leq s \leq 1,~~ s\leq t$};\\
\beta, & \mbox{if $\eta\leq s \leq  1,~~ s\geq t$}.
\end{array}
\right.\]
\end{lemma}
\begin{lemma}\cite{SZY}
The function $G(t,s)$ arising in Lemma \ref{lem1} satisfies the following conditions:
\begin{enumerate}
\item[i)] $G(t,s)$ is a continuous map defined on $[0,1]\times [0,1]$;
\item[ii)] for $t,s \in (0,1)$, we have $G(t,s)>0$ .
\end{enumerate} 
\end{lemma}
Now we prove the following lemma.
\begin{lemma}
The Green's function $G(t,s)$ derived in Lemma \ref{lem1} satisfies
$$\displaystyle \sup_{t \in [0,1]}\int_{0}^{1}G(t,s)ds=\beta + \frac{ \eta^{\alpha} }{\Gamma(\alpha + 1)}$$ and $$\displaystyle \inf_{t \in [0,1]}\int_{0}^{1}G(t,s)ds=\beta + \frac{ \eta^{\alpha} - 1}{\Gamma(\alpha + 1)}.$$ 
\end{lemma}
\begin{proof}
Let us consider the function $\varphi $ defined on $[0,1]$ by 
$$\varphi(t)=\int_{0}^{1}G(t,s)ds $$ for all $t \in [0,1]$.

Now, for $t \in [0,1]$ and $t \leq \eta, ~~s\geq \eta,$  we have $t \leq s$ and thus,
\begin{eqnarray*}
 \varphi(t)& = & \int_{0}^{1}G(t,s)ds\\
 & = & \int_{0}^{\eta}G(t,s)ds + \int_{\eta}^{1}G(t,s)ds\\
 & = & \int_{0}^{t}G(t,s)ds + \int_{t}^{\eta}G(t,s)ds + \int_{\eta}^{1}G(t,s)ds\\
 & = & \int_{0}^{t}\{\beta - \frac{(t-s)^{\alpha - 1}}{\Gamma(\alpha)} + \frac{(\eta-s)^{\alpha - 1}}{\Gamma(\alpha)}\}ds + \int_{t}^{\eta}\{\beta + \frac{(\eta-s)^{\alpha - 1}}{\Gamma(\alpha)}\}ds + \int_{\eta}^{1}\beta ds\\
 & =& \beta + \frac{ \eta^{\alpha} - t^{\alpha}}{\Gamma(\alpha + 1)}.
\end{eqnarray*}
Again, for $t \in [0,1]$ and $t \geq \eta, ~~s\leq \eta,$  we have $t \geq s$ and so, 
\begin{eqnarray*}
 \varphi(t)& = & \int_{0}^{1}G(t,s)ds\\
 & = & \int_{0}^{\eta}G(t,s)ds + \int_{\eta}^{1}G(t,s)ds\\
 & = & \int_{0}^{\eta}G(t,s)ds + \int_{\eta}^{t}G(t,s)ds + \int_{t}^{1}G(t,s)ds\\
 & = & \int_{0}^{\eta}\{\beta - \frac{(t-s)^{\alpha - 1}}{\Gamma(\alpha)} + \frac{(\eta-s)^{\alpha - 1}}{\Gamma(\alpha)}\}ds + \int_{\eta}^{t}\{\beta - \frac{(t-s)^{\alpha - 1}}{\Gamma(\alpha)}\}ds + \int_{t}^{1}\beta ds\\
 & =& \beta + \frac{ \eta^{\alpha} - t^{\alpha}}{\Gamma(\alpha + 1)}.
\end{eqnarray*}
Thus, from the above calculations we get, $$\varphi(t)=\beta + \frac{ \eta^{\alpha} - t^{\alpha}}{\Gamma(\alpha + 1)}$$ for all $ t \in [0,1]$.

Therefore, $$ \varphi'(t)=\frac{-\alpha t^{\alpha -1}}{\Gamma(\alpha + 1)}<0$$ for all $t\in [0,1].$

This implies that the function $\varphi $ is a decreasing function on $[0,1]$. So, 
\begin{eqnarray*}
\displaystyle \sup_{t \in [0,1]}\int_{0}^{1}G(t,s)ds & = & \displaystyle \sup_{t\in [0,1]} \varphi(t)\\
& = & \phi(0)\\
& = & \beta + \frac{ \eta^{\alpha} }{\Gamma(\alpha + 1)},
\end{eqnarray*}
and 
\begin{eqnarray*}
\displaystyle \inf_{t \in [0,1]}\int_{0}^{1}G(t,s)ds & = & \displaystyle \inf_{t\in [0,1]} \varphi(t)\\
& = & \phi(1)\\
& = & \beta + \frac{ \eta^{\alpha} - 1}{\Gamma(\alpha + 1)}.
\end{eqnarray*}
This completes the proof of the lemma.
\end{proof}
\begin{lemma} \label{lmn4}
For the Green's function  $G(t,s)$ derived in Lemma \ref{lem1}  $$G(t,s) \leq \beta + \frac{\eta^{\alpha - 1}}{\Gamma(\alpha)}$$ for all $t,s \in [0,1]$ holds.
\end{lemma}
\begin{proof}
From the formulation of $G(t,s)$ we get
\[\frac{\partial G(t,s)}{\partial t} = \left\{ \begin{array}{lr}
- \frac{(\alpha - 1)(t-s)^{\alpha - 2}}{\Gamma(\alpha)} , & \mbox{if $0\leq s \leq \eta,~~ s\leq t$};\\
0, & \mbox{if $0\leq s \leq \eta,~~ s\geq t$};\\
 - \frac{(\alpha - 1)(t-s)^{\alpha - 2}}{\Gamma(\alpha)}, & \mbox{if $\eta\leq s \leq 1,~~ s\leq t$};\\
0, & \mbox{if $\eta\leq s \leq  1,~~ s\geq t$}.
\end{array}
\right.\]
Therefore, for any fixed $s \in [0,1]$, we have
$$\frac{\partial G(t,s)}{\partial t} \leq 0$$ for each $t \in [0,1]$ and thus $G(t,s)$ is a decreasing function of $t$ on $[0,1]$ for each fixed $s \in [0,1].$

Thus, 
\begin{equation} \label{eqq1}
G(t,s)  \leq  G(0,s) \mbox{  for all } t,s \in [0,1]
\end{equation}
where
\[ G(0,s) =\left\{ \begin{array}{lr}
\beta + \frac{(\eta - s)^{\alpha - 1}}{\Gamma(\alpha)}, & \mbox{if $0\leq s \leq \eta$};\\
\beta, & \mbox{if $\eta \leq s \leq 1$}.\\
\end{array}
\right.\]
Therefore,
\[ \frac{\partial G(0,s)}{\partial s} =\left\{ \begin{array}{lr}
\frac{- (\alpha - 1)(\eta - s)^{\alpha - 2}}{\Gamma(\alpha)}, & \mbox{if $0\leq s \leq \eta$};\\
0, & \mbox{if $\eta \leq s \leq 1$}.\\
\end{array}
\right.\]
This shows that $\frac{\partial G(0,s)}{\partial s} \leq 0$ for all $s \in [0,1]$ and so $G(0,s)$ is a decreasing function of $s$ on $[0,1].$ Thus,
\begin{equation} \label{eqq2}
G(0,s) \leq G(0,0) = \beta + \frac{\eta^{\alpha - 1}}{\Gamma(\alpha)} \mbox{ for  all  } s\in [0,1].
\end{equation} 
From equations \ref{eqq1} and \ref{eqq2} we get
$$G(t,s) \leq \beta + \frac{\eta^{\alpha - 1}}{\Gamma(\alpha)}$$ for all $t,s \in [0,1].$
\end{proof}
Now we prove the following theorem concerning the existence and uniqueness of a positive solution to the fractional thermostat model.
\begin{theorem} \label{thm4}
Let us consider the fractional thermostat model with parameter $\lambda>0$ given by equations \ref{equ2} and \ref{equ3}. Assume that the following conditions hold:
\begin{enumerate}
\item[$(i)$] $\beta \Gamma(\alpha + 1)+ \eta^{\alpha}>1;$
\item[($ii)$] for all $s\in [0,1]$, 
 $$\lambda |f(s,u(s))-f(s,v(s))| \leq \lambda |f(s,u(s))| - \lambda \displaystyle \sup_{t\in [0,1]} |v(t)| -\psi( \displaystyle \sup_{t\in [0,1]} |u(t) - v(t)|)$$
 for some altering distance function $\psi $ and for all real valued continuous functions $u(s), v(s)$ defined on $[0,1];$
\item[($iii)$] $f$ is non-decreasing with respect to the second argument and there exists $t_0 \in (0,1)$ such that $f(t_0,0)>0$. 
\end{enumerate}
Then the fractional thermostat model with parameter $\lambda$ given by equations \ref{equ2} and \ref{equ3} has a unique positive solution for $\lambda \geq \frac{1}{k}$, where $ k= \beta + \frac{ \eta^{\alpha} - 1}{\Gamma(\alpha + 1)}.$
\end{theorem}
\begin{proof}
Consider the Banach space $C[0,1]$ of all real-valued continuous functions defined on $[0,1]$ equipped with the sup norm.

Define a mapping $T:C[0,1] \to C[0,1]$ by
$$Tu(t)=\lambda \int_{0}^{1}G(t,s)f(s,u(s))ds$$ for all $u \in C[0,1]$, where $G(t,s)$ is defined as in Lemma \ref{lem1}.

From Lemma \ref{lem1}, it is obvious  that the thermostat model \ref{equ2} and \ref{equ3} has $u(t)$ as a solution if and only if $u(t)$ is a fixed point of $T$.

Now, by condition $(ii)$ we have,
\begin{eqnarray*}
\lambda |f(s,u(s))-f(s,v(s))| &\leq & \lambda |f(s,u(s))| - \lambda \displaystyle \sup_{t\in [0,1]} |v(t)| -\psi(\displaystyle \sup_{t\in [0,1]} |u(t) - v(t)|)\\
 & = & \lambda |f(s,u(s))| - \lambda \|v\|  -\psi(\|u-v\|).
\end{eqnarray*}
Multiplying both sides by $|G(t,s)|$, we get
\begin{eqnarray*}
\lambda |f(s,u(s))-f(s,v(s))||G(t,s)| &\leq & \lambda |f(s,u(s))| |G(t,s)| - \lambda \|v\||G(t,s)|\\ & &-\psi(\|u-v\|)|G(t,s)|\\
\Rightarrow \lambda\int_{0}^{1} |f(s,u(s))-f(s,v(s))| |G(t,s)|ds & \leq & \lambda\int_{0}^{1} |f(s,u(s))| |G(t,s)|ds \\ & &- \lambda \int_{0}^{1}\|v\|  |G(t,s)| ds-\int_{0}^{1}\psi(\|u-v\|)|G(t,s)|ds\\
& = & \lambda\int_{0}^{1} G(t,s)f(s,u(s)) ds - \lambda \|v\|\int_{0}^{1}  G(t,s) ds \\ & & - \psi(\|u-v\|)\int_{0}^{1}G(t,s)ds\\
& \leq & \lambda\int_{0}^{1} G(t,s)f(s,u(s)) ds   - \lambda \|v\|\displaystyle \inf_{t\in [0,1]}\int_{0}^{1}  G(t,s) ds \\ && - \psi(\|u-v\|)\displaystyle \inf_{t\in [0,1]}\int_{0}^{1}G(t,s)ds\\
& \leq & \lambda\int_{0}^{1} G(t,s)f(s,u(s)) ds  -  \lambda k\|v\| - k\psi(\|u-v\|).\\
\end{eqnarray*}
Now if $\lambda k \geq 1$, then  from the above in-equation we obtain,
\begin{eqnarray}\label{h}
 \lambda\int_{0}^{1} |f(s,u(s))-f(s,v(s))| |G(t,s)|ds  & \leq & \lambda\int_{0}^{1} G(t,s)f(s,u(s)) ds - \|v\|  - k\psi(\|u-v\|)\nonumber \\
& \leq & \lambda\int_{0}^{1} G(t,s)f(s,u(s)) ds - |v(t)| - k\psi(\|u-v\|)\nonumber \\
& \leq &   |\lambda \int_{0}^{1} G(t,s)f(s,u(s)) ds - v(t)|\nonumber\\
& & - k\psi(\|u-v\|).
\end{eqnarray}

Therefore using Equation \ref{h} we get,
\begin{eqnarray*}
|Tu(t) - Tv(t)|& = & |\lambda \int_{0}^{1}G(t,s)f(s,u(s))ds - \lambda \int_{0}^{1}G(t,s)f(s,v(s))ds|\\
& =& |\lambda \int_{0}^{1}G(t,s)(f(s,u(s))- f(s,v(s)))ds|\\
& \leq &  |\lambda \int_{0}^{1} G(t,s)f(s,u(s)) ds - v(t)| - k\psi(\|u-v\|).\\ 
\end{eqnarray*}
The above inequality holds for all $t\in [0,1]$ and so we have,
\begin{eqnarray}
\displaystyle \sup_{t\in [0,1]}|Tu(t) - Tv(t)|& \leq & \displaystyle \sup_{t\in [0,1]}\lambda |\int_{0}^{1} G(t,s)f(s,u(s)) ds - v(t)| - k\psi(\|u-v\|) \nonumber\\ 
\Rightarrow \|Tu - Tv\| & \leq & \|Tu - v\|  - k\psi(\|u-v\|).\label{eqn4}
\end{eqnarray}
It is easily perceived by condition $(i)$ that, $k>0.$

Define two functions $\phi, \psi_1: [0,\infty)\to [0,\infty)$ by 
$$\phi(t)=t ~~ and $$ 
$$ \psi_1(t)=k\psi(t)$$
for all $t \in [0,\infty).$ Then one can easily verify that $\phi, \psi_1$ are two altering distance functions and also from equation \ref{eqn4} we get,
\begin{equation} \label{eqn5}
 \phi(\|Tu - Tv\|)  \leq  \phi(\|Tu - v\|)  - \psi_1(\|u-v\|).
\end{equation}
The above inequality holds for all $u,v \in C[0,1]$ and so by Theorem \ref{thm2}, $T$ has a unique fixed point $u(t)$, say, in $C[0,1]$.

Note that, Equation \ref{eqn5} holds if $\lambda k \geq 1$. So, $T$ has $u(t)$ as a fixed point if $\lambda k \geq 1$, i.e.,  $u(t)$ is a solution of the thermostat model \ref{equ2} and \ref{equ3} if $\lambda k \geq 1$, i.e., $\lambda \geq  \frac{1}{k}$.

Now we have $\lambda >0,~~ G(t,s) >0$ and $f(s,u(s)) \geq 0$ for all $t,s \in [0,1]$. Therefore it is clear that $$\lambda  \int_{0}^{1} G(t,s)f(s,u(s))ds \geq 0$$ for all $t\in [0,1]$. This means that $Tu(t) \geq 0$ for all $t\in [0,1]$ and which leads us to the fact that $u(t) \geq 0$ for all $t\in [0,1].$

Finally, we show that the unique solution $u(t)$ is always positive. To show this, first we show that the zero function $0$ is not a fixed point of $T$.

Suppose in contrary, assume that the zero function $0$ is a fixed point of $T$. Then, we have $$0=\lambda \int_{0}^{1}G(t,s)f(s,0)ds,$$ for all $t \in[0,1].$
Since $G(t,s)f(s,0) \geq 0$ for all $t\in [0,1]$ and for all $s\in [0,1]$, we have
$$G(t,s)f(s,0) = 0,$$ for all $t\in [0,1]$ and for almost all $s\in [0,1]$. 
This fact leads us to 
\begin{equation} \label{eqn6}
f(s,0)=0 \mbox{ for  almost  all  } s\in [0,1].
\end{equation}
By condition $(iii)$, there exists $t_0 \in (0,1)$ such that $f(t_0,0)>0$. Again, since $f$ is continuous at $(t_0,0)$, there exists a subset $A$ of $[0,1]$ of positive Lebesgue measure such that $f(s,0)>0$ for all $s\in A$. This is a contradiction to \ref{eqn6}. So the zero function $0$ is not a fixed point of $T$.

Now, let $u(t_1)=0$ for some $t_1\in (0,1)$. Therefore we have, 
\begin{eqnarray}\label{eq5}
\int_{0}^{1}G(t_1,s)f(s,u(s))ds &= & 0.
\end{eqnarray}
But $u(s) \geq 0$ for all $s\in [0,1]$ and $f$ is non-decreasing with respect to the second argument. Hence
\begin{eqnarray}\label{eq6}
0 \geq \int_{0}^{1}G(t_1,s)f(s,u(s))ds & \geq & \int_{0}^{1}G(t_1,s)f(s,0)ds \geq 0.
\end{eqnarray}
Therefore from \ref{eq5} and \ref{eq6}, we obtain
$$ \int_{0}^{1}G(t_1,s)f(s,0)ds = 0.$$
As $G(t_1,s)f(s,0) \geq 0 $, it follows that $G(t_1,s)f(s,0) = 0$ for almost all $s\in [0,1]$. This implies that $f(s,0) = 0 $ for almost all $s\in [0,1]$, which is a contradiction.

Hence it follows that, $u(t)>0$ for all $t\in (0,1).$ Again, since $u$ is continuous on $[0,1]$, we have
$u(t)>0$ for all $t\in [0,1].$
Thus the fractional thermostat model, given by Equations \ref{equ2} and \ref{equ3}, has a unique positive solution for $\lambda \geq \frac{1}{k}$, where $k=  \beta + \frac{ \eta^{\alpha} - 1}{\Gamma(\alpha + 1)}.$
\end{proof}
\begin{theorem}\label{thm5}
Let us consider the fractional thermostat model with parameter $\lambda$ given by equations \ref{equ2} and \ref{equ3}. Assume that the following conditions hold:
\begin{enumerate}
\item[$(i)$] $\beta \Gamma(\alpha + 1)+ \eta^{\alpha}>1;$
\item[$(ii)$] for all $s\in [0,1]$,
 $$\lambda |f(s,u(s))-f(s,v(s))| \leq \lambda |f(s,u(s))| - \lambda \displaystyle \sup_{t\in [0,1]} |v(t)| -\psi( \displaystyle \sup_{t\in [0,1]} |u(t) - v(t)|),$$
for some altering distance function $\psi$,
for all $u(s), v(s)$ in the set $C= \{u(s)\in C[0,1]: 0\leq |u(s)| \leq R,\mbox{ for all  }~~s\in [0,1] \mbox { and  $R$ is a positive constant}\}$,
\item[$(iii)$]$\displaystyle \int_{0}^{1} f(s,R) ds \leq \frac{R}{\lambda k_1}$, where $k_1 = \beta + \frac{\eta^{\alpha - 1}}{\Gamma(\alpha)};$
\item[$(iv)$] $f$ is non-decreasing with respect to the second argument and there exists $t_0 \in (0,1)$ such that $f(t_0,0)>0$. 
\end{enumerate}
Then the fractional thermostat model with parameter $\lambda$ given by equations \ref{equ2} and \ref{equ3} has a unique positive solution in $C$ for $\lambda  \geq \frac{1}{k} $ where $k =\beta + \frac{ \eta^{\alpha} - 1}{\Gamma(\alpha + 1)}.$  
\end{theorem}
\begin{proof}
Let us take the Banach space $C[0,1]$ endowed with the sup norm. Then it is easily noticeable that $C$ is a closed subset of $C[0,1]$.

Now, for any $u(s) \in C[0,1]$ we have $$ u(s) \leq R.$$ 
The fact that $f$ is non-decreasing with respect to the second argument gives us
\begin{eqnarray*}
\displaystyle \int_{0}^{1} f(s,u(s)) ds &\leq & \displaystyle \int_{0}^{1} f(s,R)) ds \\
& \leq & \frac{R}{\lambda k_1}.
\end{eqnarray*}
Therefore,
\begin{eqnarray*}
\lambda \int_{0}^{1} G(t,s) f(s,u(s))ds & \leq & \lambda  \big( \beta + \frac{\eta^{\alpha - 1}}{\Gamma(\alpha)} \big) \displaystyle \int_{0}^{1} f(s,u(s)) ds \\
& \leq & \lambda k_1 \frac{R}{\lambda k_1},
\end{eqnarray*}
i.e., 
\begin{equation} \label{e1}
\lambda \int_{0}^{1} G(t,s) f(s,u(s)) \leq R.
\end{equation}
Next, we define a mapping $T:C \to C$  by
$$Tu(t)=\lambda \int_{0}^{1}G(t,s)f(s,u(s))ds$$ for all $u \in C[0,1]$, where $G(t,s)$ is given by Lemma \ref{lem1}.

From Equation \ref{e1} one can easily check that $T$ is well-defined on $C$. 

We now define two functions  $ \phi, \psi_1 :[0,\infty) \to [0,\infty)$  by 
\begin{eqnarray*}
\phi (t) = t\\
\psi_1(t) = k \psi(t)
\end{eqnarray*}
for all $t \in [0, \infty)$. Then it is clear that $\phi, \psi_1$ are altering distance functions.

Now proceeding as in Theorem \ref{thm4} we get 
$$\phi(\|Tu - Tv\|)  \leq  \phi(\|Tu - v\|)  - \psi_1(\|u-v\|)$$
for all $u,v \in C$ if $\lambda \geq \frac{1}{k}.$

Thus we see that all conditions of Theorem \ref{thm2} are satisfied if $\lambda  \geq \frac{1}{k} $. So by the theorem, $T$ has a unique fixed point in $C$, say, $u(t)$.

Thus, $u(t)$  is the unique solution of the fractional thermostat model given by Equations \ref{equ2} and \ref{equ3}, which follows by Lemma \ref{lem1} and the definition of $T$. The fact that $u(t)$  is positive on $[0,1]$ follows by Theorem \ref{thm4} using condition $(iv)$. Hence,  the fractional thermostat model given by Equations \ref{equ2} and \ref{equ3} satisfying the hypotheses of Theorem \ref{thm5}, has a unique positive solution for $\lambda  \geq \frac{1}{k}.$ 
\end{proof}
Now, we demonstrate an example which validates the effectiveness of the aforementioned result.
\begin{example}
Let us consider the fractional thermostat model 
\begin{equation} \label{eqn8}
^CD^{\alpha}u(t) + \lambda f(t,u(t))=0, t\in (0,1),
 \end{equation}
\begin{equation} \label{eqn9}
u'(0)=0, \beta ^CD^{\alpha - 1}u'(1) + u(\eta)=0.
\end{equation}
We choose, $$\alpha = \frac{3}{2},~~ \beta= \frac{4}{5}, ~~ \eta=\frac{1}{2} \mbox{ and } f(t,u(t))= \ln (3^{20} + t^2) + t^3 + \frac{1}{24 - u(t)}.$$
Then, $\beta \Gamma(\alpha)-(1-\eta)^{(\alpha -1)} = \frac{4}{5}.\frac{1}{2}. \sqrt{\pi} -(\frac{1}{2})^\frac{1}{2} >0.$

Clearly $f:[0,1] \times \mathbb{R} \to \mathbb{R}^+ $ is a continuous function and also $f$ is non-decreasing with respect to the second argument and there exists $ \frac{1}{2} \in (0,1)$ such that $f(\frac{1}{2},0) >0 $.

We take $$C=\{u(s) \in C[0,1] : 0\leq |u(s)| \leq 20, \mbox{  for all } s\in [0,1]\}$$ i.e., here $R=20$.

Now,
$$\beta \Gamma(\alpha + 1) + \eta^{\alpha} = \frac{4}{5}. \frac{3}{2}. \frac{1}{2}. \sqrt\pi + (\frac{1}{2})^\frac{3}{2} \approx 1.4165 >1,$$
and 
\begin{eqnarray*}
k & = & \beta + \frac{\eta^{\alpha} - 1}{\Gamma(\alpha + 1)}\\
& = & \frac{4}{5} +\frac{(\frac{1}{2})^\frac{3}{2} -1}{\frac{3}{2}. \frac{1}{2}. \sqrt\pi}\\
& \approx & 0.3135\\
\Rightarrow \frac{1}{k} & \approx & 3.1897. 
\end{eqnarray*}
\begin{eqnarray*}
k_1 & = & \beta + \frac{\eta^{\alpha} - 1}{\Gamma(\alpha )}\\
& = & \frac{4}{5} +\frac{(\frac{1}{2})^\frac{1}{2} }{ \frac{1}{2}. \sqrt\pi}\\
& \approx & 1.5981.
\end{eqnarray*}
We choose $\lambda = 3.2$ and clearly $\lambda \geq \frac{1}{k}.$

Now, 
\begin{eqnarray*}
\displaystyle \int_{0}^{1} f(s,R)ds & \leq & \frac{1}{1 + 3^{20}} - \frac{1}{3^{20}} + \frac{1}{4} + \frac{1}{4}\\
& \leq & \frac{R}{\lambda k_1}\\
& \approx & \frac{20}{3.2 \times 1.5981}\\
& \approx & 3.9109.
\end{eqnarray*}
%
Next, we define a mapping $\psi:[0,\infty) \to [0,\infty)$ by
\[\psi(t)= \left\{\begin{array}{lr}
t^2, & \mbox {if $0 \leq t <1 $};\\
1, & \mbox{ if $ t \geq 1 $}.
\end{array}
\right.\]
Then it is an easy task to note that $\psi$ is an altering distance function.

Finally, for any $u(s), v(s) \in C$ we have,
\begin{eqnarray*}
\lambda |f(s,u(s)) - f(s,v(s))| &  = & 3.2\big|\frac{1}{24 - u(s)} - \frac{1}{24 - v(s)}\big|\\
& \leq & 3.2\big(\frac{1}{4} + \frac{1}{4} \big)\\
& = & 1.6.
\end{eqnarray*}
But, 
\begin{eqnarray*}
\lambda |f(s,u(s))| - \lambda \displaystyle \sup_{t\in [0,1]} |u(t)| - \psi(\displaystyle \sup_{t\in [0,1]} |u(t) - v(t)|) & \geq & 3.2\times 21 - 3.2 \times 20 -\psi(40)\\
& = & 2.20.
\end{eqnarray*}
Therefore,
$$ \lambda |f(s,u(s)) - f(s,v(s))| \leq \lambda |f(s,u(s))| - \lambda \displaystyle \sup_{t\in [0,1]} |u(t)| - \psi(\displaystyle \sup_{t\in [0,1]} |u(t) - v(t)|)$$
for all $u(s), v(s) \in C$.
So, by Theorem \ref{thm5} the thermostat model, given by Equations \ref{eqn8} and \ref{eqn9} has a unique positive solution in $C$ for $\lambda = 3.2$. 
\end{example}
\vskip.5cm\noindent{\bf Acknowledgements:}\\
The first named author would like to express his genuine appreciation to CSIR, New Delhi, India for their financial supports. Also the third named author would like to convey his cordial thanks to DST-INSPIRE, New Delhi, India for their financial aid under INSPIRE fellowship scheme.


%
\end{document}